%% file: Abelian.tex
\documentclass[12pt]{amsart}
\usepackage[utf8,utf8x]{inputenc}
\usepackage[T1]{fontenc}
\usepackage{a4wide}
\usepackage{eucal}
\usepackage{latexsym}
\usepackage{amssymb}
\usepackage{verbatim}
\usepackage{epic,eepic}
\usepackage{epsfig}
\usepackage[all]{xy}
\usepackage{listings}
\usepackage{rotating}
\usepackage{graphicx}
\usepackage{psfrag}
\usepackage[dvipsnames]{xcolor}
\usepackage{tikz}
\usepackage{multirow}
\usepackage{mdwlist}
\usepackage[pdfauthor   = {Mohamed\ Barakat\ and\ Markus\ Lange-Hegermann},
            pdftitle    = {An\ Axiomatic\ Setup\ for\ Algorithmic\ Homological\ Algebra\ and\ an\ Alternative\ Approach\ to\ Localization},
            pdfsubject  = {18E10; 18E25; 18G05; 18G10; 18G15; 13H99; 13P10; 13P20},
            pdfkeywords = {solving\ (in)homogeneous\ linear\ systems;\ constructive\ homological\ algebra;\ Abelian\ category;\ localization;\ Mora's\ algorithm;\ homalg;\ LocalizeRingForHomalg},
            bookmarks=true,
            bookmarksopen=true,
            pagebackref=true,
            hyperindex=true,
            colorlinks=true,
            linkcolor=blue,
            citecolor=blue,
            filecolor=blue,
            urlcolor=blue,
            pdftex=true
            ]{hyperref}

\newtheorem{theorem}{Theorem}[section]

\newtheorem{lemm}[theorem]{Lemma}
\newtheorem{prop}[theorem]{Proposition}
\newtheorem{coro}[theorem]{Corollary}
\theoremstyle{definition}
\newtheorem{defn}[theorem]{Definition}
\newtheorem{rmrk}[theorem]{Remark}
\newtheorem{exmp}[theorem]{Example}

\DeclareMathOperator{\coker}{coker}
\DeclareMathOperator{\img}{im}
\DeclareMathOperator{\Hom}{Hom}

\newcommand\A{\mathcal{A}}

\newcommand\B{\mathcal{B}}

\newcommand{\Q}{\mathbb{Q}}
\newcommand\tL{\mathtt{L}}
\newcommand\tM{\mathtt{M}}
\newcommand\tm{\mathtt{m}}
\newcommand\tN{\mathtt{N}}
\newcommand\tX{\mathtt{X}}
\newcommand\tx{\mathtt{x}}
\newcommand\tA{\mathtt{A}}
\newcommand\tB{\mathtt{B}}
\newcommand\tb{\mathtt{b}}
\newcommand\tC{\mathtt{C}}
\newcommand\tK{\mathtt{K}}
\newcommand\tF{\mathtt{F}}
\newcommand\tI{\mathtt{I}}
\newcommand{\m}{\mathfrak{m}}
\newcommand\tT{\mathtt{T}}
\newcommand\ttt{\mathtt{t}}
\newcommand\ts{\mathtt{s}}
\newcommand{\Rfpmod}{R-\mathbf{fpmod}}
\newcommand{\Rfpres}{R-\mathbf{fpres}}

\definecolor{darkgreen}{rgb}{0.008,0.617,0.067}
\definecolor{brown}{rgb}{0.6,0.4,0.2}

\begin{document}

\lstset{basicstyle=\scriptsize\ttfamily,
        frame=single}

\author{Mohamed Barakat}
\address{Department of mathematics, University of Kaiserslautern, 67653 Kaiserslautern, Germany}
\email{\href{mailto:Mohamed Barakat <barakat@mathematik.uni-kl.de>}{barakat@mathematik.uni-kl.de}}

\author{Markus Lange-Hegermann}
\address{Lehrstuhl B f\"ur Mathematik, RWTH-Aachen University, 52062 Germany}
\email{\href{mailto:Markus Lange-Hegermann <markus.lange.hegermann@rwth-aachen.de>}{markus.lange.hegermann@rwth-aachen.de}}

\title[An Axiomatic Setup for Algorithmic Homological Algebra, Localization]{An Axiomatic Setup for \\ Algorithmic Homological Algebra \\ and an Alternative Approach to Localization}

\dedicatory{
\textit{This paper is dedicated to our teacher Professor \textsc{Wilhelm Plesken} \\ on the occasion of his 60th birthday.}
}

\begin{abstract}
In this paper we develop an axiomatic setup for algorithmic homological algebra of \textsc{Abel}ian categories. This is done by exhibiting all existential quantifiers entering the definition of an \textsc{Abel}ian category, which for the sake of computability need to be turned into constructive ones. We do this explicitly for the often-studied example \textsc{Abel}ian category  of finitely presented modules over a so-called computable ring $R$, i.e., a ring with an explicit algorithm to solve one-sided (in)homogeneous linear systems over $R$. For a finitely generated maximal ideal $\m$ in a commutative ring $R$ we show how solving (in)homogeneous linear systems over $R_\m$ can be reduced to solving associated systems over $R$. Hence, the computability of $R$ implies that of $R_\m$. As a corollary we obtain the computability of the category of finitely presented $R_\m$-modules as an \textsc{Abel}ian category, without the need of a \textsc{Mora}-like algorithm. The reduction also yields, as a by-product, a complexity estimation for the ideal membership problem over local polynomial rings. Finally, in the case of localized polynomial rings we demonstrate the computational advantage of our homologically motivated alternative approach in comparison to an existing implementation of \textsc{Mora}'s algorithm.
\end{abstract}

\thanks{
  We would like to thank the anonymous referee for valuable comments.
}

\maketitle


\section{Introduction}

As finite dimensional constructions in linear algebra over a field $k$ boil down to \textbf{solving (in)homogeneous linear systems} over $k$, the \textsc{Gauss}ian algorithm makes the whole theory perfectly computable, provided $k$ itself is. Solving linear systems in \textsc{Gauss}ian form, i.e., in reduced echelon form, is a trivial task. And computing the \textsc{Gauss}ian form of a linear system is thus a major step towards its solution.

Homological algebra of module categories can be viewed as linear algebra over general rings. Hence, in analogy to linear algebra over a field one would expect that solving linear systems would play an important rule in making the theory computable. \cite{BR} introduced a data structure for additive functors of module categories useful for an efficient computer implementation. Solving linear systems was used to describe the calculus of such functors in a constructive way. Here we proceed in a more foundational manner. We show that solving linear systems is, as expected, the key to the complete computability of the category of finitely presented modules, \emph{merely} viewed as an \textsc{Abel}ian category \cite{HS,Rot09,weihom}.  In Section \ref{basic_constructions} we list all the existential quantifiers entering the definition of an \textsc{Abel}ian category. Turning all of them into algorithms for any given \textsc{Abel}ian category establishes its computability. This abstract point of view widens the range of applicability of a computer implementation along these lines beyond the context of module categories.

Section \ref{computability} addresses the computability of the \textsc{Abel}ian categories of finitely presented modules over so-called computable rings. A ring $R$ is called \textbf{computable} if one can effectively solve (in)homogeneous linear systems over $R$ (cf.~Def.~\ref{computable}). Proposition~\ref{coker} together with Theorem~\ref{R-fpres_basic_constructions} show that, as expected, the computability of the ring together with some simple matrix operations indeed suffice to provide all the algorithms needed to make this category computable \emph{as an \textsc{Abel}ian category}.

One way to verify the computability of a ring is to find an appropriate substitute for the \textsc{Gauss}ian algorithm. Fortunately such substitutes exist for many rings of interest. Beside the well-known \textsc{Hermite} normal form algorithm for principal ideal rings with computable gcd's, it turns out that appropriate generalizations of the classical \textsc{Gröbner} basis algorithm for polynomial rings \cite{Buch} provide the desired substitute for a wide class of commutative \emph{and} noncommutative rings \cite{LevThesis,RobPhd}.

Although finding a substitute for the \textsc{Gauss}ian algorithm, which we will refer to as computing a \textbf{``distinguished basis\footnote{Basis in the sense of a generating set, and \emph{not} in the sense of a free basis.}''}, is the traditional way to solve linear systems over rings, it is only \emph{one} mean to this end. Indeed, other means do exist: \\
Let $R_\m$ be the localization of the commutative ring $R$ at a finitely generated maximal ideal $\m \triangleleft R$. Theorem~\ref{ThmLocal} together with Corollary~\ref{CoroLocal} show how computations in the \textsc{Abel}ian category of finitely presented modules over the local ring $R_\m$ can be reduced to computations over the global ring $R$. In particular, one can avoid computing distinguished bases over the local ring $R_\m$. The idea is very simple. Elements of the local ring $R_\m$ can be viewed as numerator-denominator pairs $(n,d)$ with $r\in R$ and $d\in R\setminus \m$. Likewise, $(r\times c)$-matrices over $R_\m$ can be viewed as numerator-denominator pairs $(N,d)$ with $N\in R^{r\times c}$ and $d\in R\setminus\m$. It is now easy to see that solving (in)homogeneous linear systems over $R_\m$ can simply be done by solving associated systems over $R$, thus deducing the computability of $R_\m$ from that of $R$.

In principle, \textsc{Mora}'s algorithm, which provides a ``distinguished basis'', can replace \textsc{Buchberger}'s algorithm for all sorts of computations over localized polynomial rings. This also seems to be the common practice. Nevertheless, a considerable amount of these computations only depend on the category of finitely presented modules over localized polynomial rings \emph{merely} being \textsc{Abel}ian. From this point of view we show how \textsc{Buchberger}'s algorithm suffices to carry out all such constructions and explain in §\ref{comparison mora} why our homologically motivated approach to localization of polynomial rings is computationally superior to an approach based on \textsc{Mora}'s algorithm. However, \textsc{Mora}'s algorithm remains indispensable when it comes to the computation of \textsc{Hilbert} series of modules over such rings, for example. Still, these modules are normally the outcome of huge homological computations, which often enough only become feasible through our alternative approach. \textsc{Serre}'s intersection formula is a typical example of this situation (cf.~Example~\ref{IntersectionFormula}).

In Section \ref{implementation} we will shortly describe our implementation. The examples in Section \ref{examples} illustrate the computational advantage of our alternative approach. An existing performant implementation of \textsc{Mora}'s algorithm fails for these examples.

The paper suggests a specification which can be used to realize a constructive setting for the homological algebra of further concrete \textsc{Abel}ian categories. Realizing this for other \textsc{Abel}ian categories is work in progress.

\section{Basic Constructions in \textsc{Abel}ian Categories}\label{basic_constructions}

The aim of this section is to list the basic constructions of an \textsc{Abel}ian category with enough projectives, which suffice to build all the remaining ones. In case these few basic constructions are computable, it follows that all further constructions become computable as well.

In the list we only want to emphasize the existential quantifiers, that need to be turned into constructive ones. We decided to suppress the universal properties needed to correctly formulate some of the points below, as we assume that they are well-known to the reader.

\bigskip
\noindent
$\A$ is a \textbf{category}:
\begin{enumerate}
  \item\label{Ab_IdentityMatrix} For any object $M$ there exists an \textbf{identity morphism} $1_M$.
  \item\label{Ab_Compose} For any two composable morphisms $\phi,\psi$ there exists a \textbf{composition} $\phi\psi$.
\suspend{enumerate}
$\A$ is a category \textbf{with zero}:
\resume{enumerate}
  \item There exists a \textbf{zero object} $0$.
  \item\label{Ab_ZeroMatrix} For all objects $M,N$ there exists a \textbf{zero morphism} $0_{MN}$.
\suspend{enumerate}
$\A$ is an \textbf{additive} category:
\resume{enumerate}
  \item\label{Ab_AddMat} For all objects $M,N$ there exists an \textbf{addition}\footnote{In fact, the addition can be recovered from the product, coproduct, composition, and identity morphisms. A direct description of the addition is nevertheless important for computational efficiency.} $(\phi,\psi)\mapsto \phi+\psi$ in the \textsc{Abel}ian group $\Hom_\A(M,N)$.
  \item\label{Ab_SubMat} For all objects $M,N$ there exists a \textbf{subtraction} $(\phi,\psi)\mapsto \phi-\psi$ in the \textsc{Abel}ian group $\Hom_\A(M,N)$.
  \item\label{Ab_DiagMat} For all objects $A_1,A_2$ there exists a \textbf{direct sum} $A_1\oplus A_2$.
  \item\label{Ab_UnionOfRows} For all pairs of morphisms $\phi_i:A_i\to M$, $i=1,2$ there exists a \textbf{coproduct morphism} $\langle\phi_1,\phi_2\rangle:A_1\oplus A_2\to M$.
  \item\label{Ab_UnionOfColumns} For all pairs of morphisms $\phi_i:M\to A_i$, $i=1,2$ there exists a \textbf{product morphism} $\{\phi_1,\phi_2\}:M\to A_1\oplus A_2$.
\suspend{enumerate}
$\A$ is an \textbf{\textsc{Abel}ian} category:
\resume{enumerate}
  \item\label{Ab_SyzygiesGenerators} For any morphism $\phi:M\to N$ there exists a \textbf{kernel} $\ker \phi \stackrel{\kappa}{\hookrightarrow} M$, such that
  \item\label{Ab_DecideZeroEffectively} for any morphism $\tau:L\to M$ with $\tau\phi=0$ there exists a \emph{unique} \textbf{lift} $\tau_0: L\to \ker \phi$ of $\tau$ along $\kappa$, i.e., $\tau_0 \kappa = \tau$ (cf.~\ref{lift_colift}.(k).(ii)).
  \item\label{Ab_UnionOfRows_IdentityMatrix} For any morphism $\phi:M\to N$ there exists a \textbf{cokernel} $N \stackrel{\epsilon}{\twoheadrightarrow} \coker \phi$, such that
  \item\label{Ab_colift} for any morphism $\eta:N\to L$ with $\phi\eta=0$ there exists a \emph{unique} \textbf{colift} $\eta_0: \coker \phi \to L$ of $\eta$ along $\epsilon$, i.e., $\epsilon \eta_0 = \eta$ (cf.~\ref{lift_colift}.(c).(ii)).
\suspend{enumerate}
$\A$ has \textbf{enough projectives}:
\resume{enumerate}
  \item\label{projective_lift} $P$ is called projective if for each morphism $\phi:P\to N$ and each morphism $\epsilon:M \to N$ with $\img \phi\leq \img \epsilon$ there exists a \textbf{projective lift} $\phi_1:P\to M$ of a $\phi$ along $\epsilon$ (cf.~Def.~\ref{special_projective_lift} and Remark~\ref{general_projective_lift}).
  \item\label{projective_hull} For each object $M$ there exists a \textbf{projective hull} $\nu: P \twoheadrightarrow M$.
\end{enumerate}

\bigskip
If for an \textsc{Abel}ian category $\A$ we succeed in making the above basic constructions computable, all further constructions which only depend on $\A$ being \textsc{Abel}ian will be computable as well.

\begin{defn}
  Let $\A$ be an \textsc{Abel}ian category.
  \begin{enumerate}
    \item We say that $\A$ is \textbf{computable as an \textsc{Abel}ian category} if the existential quantifiers in (\ref{Ab_IdentityMatrix})-(\ref{Ab_colift}) can be turned into constructive ones.
    \item If additionally the existential quantifiers in (\ref{projective_lift})-(\ref{projective_hull}) can be turned into constructive ones, then we say that $\A$ is \textbf{computable as an \textsc{Abel}ian category with enough projectives}.
    \end{enumerate}
\end{defn}

\cite{BaSF} details a construction of spectral sequences (of filtered complexes) only using the axioms of an \textsc{Abel}ian category as detailed above. In particular, all arguments are based on operations on morphisms rather than chasing single elements.

To compute the derived functors of an additive functor $F:\A \to \B$ where $\A$ does not have enough projectives resp.\  injectives one needs to provide a substitute for projective resp.\  injective resolutions. The abstract \textsc{de Rham} theorem suggests the use of so-called left (resp.\  right) $F$-acyclic resolutions (as used in \cite[Prop.~III.6.5]{Har}, for example).

\section{Computability in \textsc{Abel}ian categories of finitely presented modules}\label{computability}

The previous section suggests a short path to ensure computability in an \textsc{Abel}ian category. This section follows that path for the often-studied example of module categories.

\smallskip
From now on let $\A:=\Rfpmod$ be the category of finitely presented \emph{left} $R$-modules.
The category $\mathbf{fpmod}-R$ of finitely presented right $R$-modules is treated analogously. In this section we will show how to make the basic operations of §\ref{basic_constructions} computable. As customary from linear algebra the basic data structure for computations will be finite dimensional matrices over $R$.

Finitely presented $R$-modules are in particular finitely generated. Thus, a morphism in $\Rfpmod$ can be represented by a finite dimensional matrix, the so-called \textbf{representation matrix}, with entries in $R$.

A free object in $\Rfpmod$ is a free module of finite rank $r$, i.e., a module of the form $R^{1\times r}$. And since every finitely generated module $M$ is an epimorphic image of some $\nu:F_0 \twoheadrightarrow M$ with $F_0=R^{1\times r_0}$, it follows that $\Rfpmod$ even has \textbf{enough free objects}.

By definition of $\Rfpmod$ each object even admits an exact sequence $\xymatrix@1{F_1 \ar[r]^{\partial_1} & F_0 \ar@{>>}[r]^{\nu} & M}$, $F_1=R^{1\times r_1}$ being free of finite rank. The morphism $\partial_1$ is called a \textbf{finite free presentation} of $M$. If we denote by $\tM\in R^{r_1\times r_0}$ the matrix representing $\partial_1$ and call it \textbf{presentation matrix} of $M$, then $\nu$ induces an isomorphism from
\[
  \coker \tM := R^{1\times r_0}/R^{1\times r_1}\tM = \coker( R^{1\times r_1} \xrightarrow{\tM} R^{1\times r_0} ) = \coker\partial_1
\]
to $M$. The rows of $\tM$ are regarded as \textbf{relations} among the $r_0$ \textbf{generators} of $M$ given by the residue classes of the unit row vectors in $R^{1\times r_0}/R^{1\times r_1}\tM$ (cf.~\cite[§2]{BR}, \cite[Def.~2.1.23]{GP08}, \cite[Def.~1.11]{DL}).

Denote by $\Rfpres$ the category of finite left $R$-presentations with objects being finite dimensional matrices over $R$, where one identifies two matrices $\tM\in R^{r_1\times r_0}$  and $\tM'\in R^{r_1'\times r_0}$ with the same number of columns to one object, if $R^{1\times r_1}\tM=R^{1\times r_1'}\tM'$, as $R$-submodules of $R^{1\times r_0}$.  The set $\Hom_{\Rfpres}(\tM,\tL)$ of morphisms between two objects $\tM\in R^{r_1\times r_0}$ and $\tL\in R^{r_1'\times r_0'}$ is the set  of all $r_0\times r_0'$-matrices  $\varphi$ over $R$ with $R^{1\times r_1}\tM\varphi\leq R^{1\times r_1'}\tL$, where one identifies two matrices $\varphi_1$ and $\varphi_2$ to one morphism, if they induce the same $R$-module homomorphism from
$\coker \tM$ to $\coker \tL$. Summing up:
\begin{prop}\label{coker}
$\xymatrix@1{\Rfpres\ \ar[r]^(.47)\coker &  \  \Rfpmod}$ is an equivalence of categories.
\end{prop}

The advantage of $\Rfpres$ is that it can directly be realized on a computer. Hence, describing the basic constructions of §\ref{basic_constructions} in $\Rfpres$ makes the category $\Rfpmod$ computable.

But we note that for $\Rfpmod$ or equivalently $\Rfpres$ to be \textsc{Abel}ian, cokernels and kernels of morphisms between finitely presented modules need to be finitely presented. This is obvious for cokernels but in general false for kernels:

\medskip
\framebox{
\begin{minipage}[c]{0.94\linewidth}
\textbf{Assumption (*)}: \\
\noindent
From now on we assume that $R$ is a ring for which the category $\Rfpmod$ is \textsc{Abel}ian.
\end{minipage}
}

\medskip
\noindent
Left \textsc{Noether}ian rings are the most prominent rings satisfying this assumption.

\subsection{Basic operations for matrices and computable rings}

Let $\tA$ be an $r_1\times r_0$-matrix over $R$.

\subsubsection{$\mathtt{(Relative)SyzygiesGenerators}$} \label{SyzygiesGeneratorsDef}
An $\tX\in R^{r_2\times r_1}$ is called a matrix of \textbf{generating syzygies (of the rows) of $\tA$} if for all $\tx\in R^{1\times r_1}$ with $\tx\tA=\mathtt{0}$, there exists a $\mathtt{y}\in R^{1\times r_2}$ such that $\mathtt{y}\tX=\tx$. The rows of $\tX$ are thus a generating set of the kernel of the map $R^{1\times r_1}\xrightarrow{\tA} R^{1\times r_0}$. We write
\[
  \tX=\mathtt{SyzygiesGenerators}(\tA)
\]
and say that $\tX$ is the most general solution of the homogeneous linear system $\tX\tA=\mathtt{0}$.

Further let $\mathtt{L}$ be an $r_1'\times r_0$-matrix over $R$. $\tX\in R^{r_2\times r_1}$ is called a matrix of \textbf{relative generating syzygies (of the rows) of $\tA$ modulo $\mathtt{L}$} if the rows of $\tX$ form a generating set of the kernel of the map $R^{1\times r_1} \xrightarrow{\tA} \coker \mathtt{L}$. We write
 \[
   \tX=\mathtt{RelativeSyzygiesGenerators}(\tA,\mathtt{L})
 \]
and say that $\tX$ is the most general solution of the homogeneous linear system $\tX\tA+\mathtt{Y}\mathtt{L}=\mathtt{0}$. This last system is of course equivalent\footnote{In practice, however, one can often implement efficient algorithms to compute $\tX$ without explicitly computing $\mathtt{Y}$.} to solving the homogeneous linear system $\begin{pmatrix}\tX & \mathtt{Y} \end{pmatrix} \begin{pmatrix} \tA \\ \mathtt{L} \end{pmatrix}=\mathtt{0}$ (cf.~\cite[§3.2]{BR}).

\subsubsection{$\mathtt{DecideZero(Effectively)}$ and $\mathtt{RightDivide}$}\label{RightDivide}
Further let $\tB$ be an $r_2\times r_0$-matrix over $R$. Deciding the solvability and solving the inhomogeneous linear system $\tX\tA=\tB$ is equivalent to the construction of matrices $\tN,\tT$ such that $\tN=\tT\tA+\tB$ satisfying the following condition: If the $i$-th row of $\tB$ is a linear combination of the rows of $\tA$, then the $i$-th row of $\tN$ is zero\footnote{So we do not require a ``normal form'', but only a mechanism to decide if a row is zero modulo some relations.}. Hence the inhomogeneous linear system $\tX\tA=\tB$ is solvable (with $\tX=-\tT$), if and only if $\tN=\mathtt{0}$. We write
\[
  (\tN,\tT)=\mathtt{DecideZeroEffectively}(\tA,\tB) \  \mbox{ and }\  \tN=\mathtt{DecideZero}(\tA,\tB).
\]
In case $\tN=\mathtt{0}$ we write $\tX=\mathtt{RightDivide}(\tB,\tA)$.

Rows of the matrices $\tA$ and $\tB$ can be considered as elements of the free module $R^{1\times r_0}$. Deciding the solvability of the inhomogeneous linear system $\tX\tA=\tB$ for a single row matrix $\tB$ is thus nothing but the \textbf{submodule membership problem} for the submodule generated by the rows of the matrix $\tA$. Finding a particular solution $\tX$ (in case one exists) solves the submodule membership problem \textbf{effectively}.

As with relative syzygies we also consider a relative version. In case the inhomogeneous system $\tX\tA=\tB \mod \tL$ is solvable, we denote a particular solution by
\[
  \tX=\texttt{RightDivide}(\tB,\tA,\tL).
\]
This is equivalent to solving $\begin{pmatrix}\tX & \mathtt{Y} \end{pmatrix} \begin{pmatrix} \tA \\ \mathtt{L} \end{pmatrix}=\tB$. For details cf.~\cite[§3.1.1]{BR}.

\begin{defn} \label{computable}
  A ring $R$ is called \textbf{left} (resp.\  \textbf{right}) \textbf{computable} if any finite dimensional inhomogeneous linear system $\tX\tA=\tB$ (resp.\  $\tA\tX=\tB$) over $R$ is effectively solvable in the following sense: There exists \emph{algorithms} computing $\mathtt{SyzygiesGenerators}(\tA)$ and $\mathtt{DecideZeroEffectively}(\tA,\tB)$. $R$ is called \textbf{computable} if it is left and right computable.
\end{defn}

In other words, a ring $R$ is computable if one can \emph{effectively solve (in)homogeneous linear systems over $R$}.

\begin{rmrk}
  If the ring $R$ is left computable then the categories $\Rfpres$ and (hence) $\Rfpmod$ are \textsc{Abel}ian, and the Assumption (*) is satisfied.
\end{rmrk}

We want to emphasize that all the free modules used in the constructions below are assumed to be given on a \emph{free set of generators}. This is necessary since there is no known algorithm to decide whether a finitely presented module over a computable ring $R$ is free or not. In practice this means that if we need to construct a free left $R$-module of rank $r$ we simply present it by the empty matrix in $R^{0\times r}$.

\subsection{Computability of the category of finite presentations}

\begin{theorem}\label{R-fpres_basic_constructions}
Let $R$ be a (left) computable ring and $\A:=\Rfpres$ the \textsc{Abel}ian category of finite left $R$-presentations. Then $\A$ is computable as an \textsc{Abel}ian category with enough projectives.
\end{theorem}
\begin{proof}
Using all of the vocabulary introduced so far we show how for the category $\Rfpres$ the \ref{projective_hull} operations for \textsc{Abel}ian categories listed in §\ref{basic_constructions} can be turned into algorithms.

In the following we denote $\tM\in R^{r_1\times r_0}$ and $\tN\in R^{s_1\times s_0}$ presentation matrices of $M$ and $N$, respectively. I.e., $M:=\coker \tM$ and $N:=\coker \tN$.

\bigskip
\noindent
$\Rfpres$ is a \textbf{category}:
\begin{enumerate}
  \item\label{IdentityMatrix} \texttt{IdentityMatrix}: The identity morphism $1_M$ of $M:=\coker(R^{1\times r_1} \xrightarrow{\tM} R^{1\times r_0})$ is represented by the identity matrix $\tI_{r_0}\in R^{r_0\times r_0}$.
  \item\label{Compose} \texttt{Compose}: The composition of two composable morphisms $\phi,\psi$ represented by the matrices $\tA,\tB$ is represented by the matrix product $\tA\tB$.
\suspend{enumerate}
$\Rfpres$ is a category \textbf{with zero}:
\resume{enumerate}
  \item A zero object $0$ is presented by an empty matrix in $R^{0\times 0}$.
  \item\label{ZeroMatrix} \texttt{ZeroMatrix}: The zero morphism $0_{MN}$ for pairs of objects $M:=\coker(R^{1\times r_1} \xrightarrow{\tM} R^{1\times r_0})$, $N:=\coker(R^{1\times s_1} \xrightarrow{\tN} R^{1\times s_0})$ is represented by the zero matrix $\mathtt{0}_{r_0,s_0}\in R^{r_0\times s_0}$.
\suspend{enumerate}
$\Rfpres$ is an \textbf{additive} category:
\resume{enumerate}
  \item\label{AddMat} \texttt{AddMat}: The addition of two morphisms $\phi,\psi:M\to N$ represented by the matrices $\tA,\tB$ is represented by the matrix sum $\tA+\tB$.
  \item\label{SubMat} \texttt{SubMat}: The difference of two morphisms $\phi,\psi:M\to N$ represented by the matrices $\tA,\tB$ is represented by the matrix subtraction $\tA-\tB$.
  \item\label{DiagMat} \texttt{DiagMat}: The direct sum of two objects $M$ and $N$ is presented by the block diagonal matrix $\begin{pmatrix} \tM & 0 \\ 0 & \tN \end{pmatrix}$.
  \item\label{UnionOfRows} \texttt{UnionOfRows}: The coproduct morphism $\langle\phi,\psi\rangle$ of two morphisms $\phi:M\to L$ and $\psi:N\to L$ represented by the matrices $\tA$ and $\tB$ is represented by the stacked matrix $\begin{pmatrix} \tA \\ \tB \end{pmatrix}$.
  \item\label{UnionOfColumns} \texttt{UnionOfColumns}: The product morphism $\{\phi,\psi\}$ of two morphisms $\phi:L\to M$ and $\psi:L\to N$ represented by the matrices $\tA$ and $\tB$ is represented by the augmented matrix $\begin{pmatrix} \tA & \tB \end{pmatrix}$.
\suspend{enumerate}
$\Rfpres$ is an \textbf{\textsc{Abel}ian} category:
\resume{enumerate}
  \item\label{SyzygiesGenerators} \texttt{RelativeSyzygiesGenerators}: To compute the kernel $\ker \phi \stackrel{\kappa}{\hookrightarrow} M$ of a morphism $\phi:M\to N$ represented by a matrix $\tA$ we do the following: First compute $\tX=\mathtt{RelativeSyzygiesGenerators}(\tA,\tN)$, the matrix representing $\kappa$. Then $\ker\phi$ is presented by the matrix $\tK=\mathtt{RelativeSyzygiesGenerators}(\tX,\tM)$.
  \item\label{DecideZeroEffectively} \texttt{DecideZeroEffectively}: Let $\tau:L\to M$ be a morphism represented by a matrix $\tB$ and $\kappa:K \hookrightarrow M$ a monomorphism represented by a matrix $\tA$ with $\tau\phi=0$ for $\phi=\coker\kappa$. Then the matrix $\tX=\texttt{RightDivide}(\tB,\tA,\tM)$ is a representation matrix for $\tau_0:L\to K$, the lift of $\tau$ along $\kappa$. It is an easy exercise to see that $\tX$ indeed represents a \emph{morphism} (cf.~\cite[3.1.1, case (2)]{BR}).
    \item\label{UnionOfRows_IdentityMatrix} \texttt{UnionOfRows} \& \texttt{IdentityMatrix}: The cokernel module $\coker \phi$ of a morphism $\phi:M\to N=\coker\tN$ represented by the matrix $\tA$ is presented by the stacked matrix
    $\begin{pmatrix} \tA \\ \tN \end{pmatrix}$. The natural epimorphism $N \stackrel{\epsilon}{\twoheadrightarrow} \coker \phi$ is represented by the identity matrix $\tI_{s_0}\in R^{s_0\times s_0}$.
  \item\label{Colift} Without loss of generality assume that the cokernel module $C=\coker\phi$ is presented according to (\ref{UnionOfRows_IdentityMatrix}), with $\tI_{s_0}$ the representation of the natural epimorphism $\epsilon:N \twoheadrightarrow C$. Further let  $\eta:N\to L$ be a morphism represented by $\tB$. Then the colift $\eta_0:C\to L$ along $\epsilon$ is again given by the matrix $\tB$.
\suspend{enumerate}
$\Rfpres$ has \textbf{enough free objects}:
\resume{enumerate}
  \item\label{FreeLift} \texttt{DecideZeroEffectively}: Let $F$ be a free $R$-module presented by an empty matrix, i.e., $F$ is given on a set of \emph{free} generators. Further let $\phi:F\to N$ and $\epsilon:M \to N$ be morphisms represented by the matrices $\tB$ and $\tA$, respectively. The image condition $\img\phi\leq\img\epsilon$ guarantees the existence of the matrix $\tX=\texttt{RightDivide}(\tB,\tA,\tN)$, which is a representation matrix of a \textbf{free lift} $\phi_1:F\to M$ along $\epsilon$ (cf.~\cite[3.1.1, case (1)]{BR}).
  \item\label{FreeHull} \texttt{IdentityMatrix}: A free hull $\nu: F\to M$ of $M$ is given by $F=\coker \tF$ with $\tF=\mathtt{0}\in R^{0\times r_0}$ and $\nu$ is represented by the identity matrix $\tI_{r_0}$.
\end{enumerate}
\end{proof}

In the constructive setting we seek, it is necessary to decide if a module is zero and if two morphisms are equal. To decide if $M=\coker \tM$ is zero check whether $\texttt{DecideZero}(\tI_{r_0},\tM)=\texttt{0}$. To decide the equality of two morphisms $\phi,\psi:M\to N = \coker \tN$ represented by $\tA$ and $\tB$ check whether $\texttt{DecideZero}(\tA-\tB,\tN)=\texttt{0}$. This, in turn, enables deciding properties like monic, epic, exactness of two composable morphisms, etc.

As $R$-modules are sets in makes sense to compute in $M=\coker \tM$. This, in turn, requires deciding equality of two elements in $m,m'\in M$, represented by two rows $\tm,\tm'\in R^{1\times r_0}$, respectively. This is again achieved by test whether $\texttt{DecideZero}(\tm-\tm',\tM)=\texttt{0}$.

\subsection{Closed symmetric monoidal \textsc{Abel}ian categories}\label{Hom}
The category of $R$-modules over a \emph{commutative} ring $R$ admits further constructions. The tensor product $M\otimes_R N$ of two $R$-modules $M,N$ turns the $\Rfpmod$ into a symmetric monoidal category. It is even a closed symmetric monoidal category with the homomorphism module $\Hom_R(M,N)$ as an internal $\Hom$ object.

For the constructibility of the tensor product and its derived functors $\operatorname{Tor}^R_i$ in $\Rfpres$ see \cite[Example.~7.1.5]{GP08} or \cite[Problem~4.7]{DL}. For the (internal) $\Hom$ module and the higher extension functors $\operatorname{Ext}_R^c$ see \cite[Example~2.1.26]{GP08}, \cite[Problems~4.5,4.6]{DL}, \cite[Thm.~3.3.15]{KR}. The morphism part of these and other functors is systematically dealt with in \cite{BR}.

\subsection{Free and projective modules}

Whereas a finitely presented module is free on free generators if an only if its presentation matrix is zero, deciding the freeness of a finitely presented module over a computable ring, let alone computing a free basis, can in general be highly non-trivial. Deciding projectiveness is often easier than deciding freeness: Let $\nu:F_0\twoheadrightarrow M$ be a free presentation of the $R$-module $M$. It follows that $M$ is projective if and only if $\nu$ admits a section $\sigma:M\hookrightarrow F_0$ (i.e., $\sigma\nu=\mathrm{id}_M$). Finding the section $\sigma$ for a finitely and freely presented module $M\stackrel{\nu}{\twoheadleftarrow} F_0 \xleftarrow{\tM} F_1$ leads to solving a two-sided inhomogeneous linear system\footnote{$\tX+\mathtt{Y}\tM=\mathtt{Id},\  \tM\tX=0$, where $\tX$ is a square matrix representing $\sigma$ and $\mathtt{Y}$ another unknown matrix.} over $R$, which can of course be brought to a one-sided inhomogeneous linear system if $R$ is commutative. Hence, testing projectiveness of finitely presented modules over commutative computable rings is constructive \cite{ZL}. Another simple exercise is to see that an $R$-module is projective if and only if $\operatorname{Ext}_R^1(M,K_1(M))=0$, where $K_1(M)$ is the first syzygy module of $M$. This can also be turned into an algorithm for commutative computable rings as outlined in §\ref{Hom}. Without the commutativity assumption \textsc{Serre}'s Remark \cite{FAC} states that a module admitting a \emph{finite free} resolution is projective if and only if it is stably free. Shortening the finite free resolution \cite[Prop.~5.11]{lam99} yields a simple proof that can be turned into an algorithm for computing the projective dimension of $M$ (and hence to decide projectiveness) whenever $M$ is a left (resp.\  right) module over the left (resp.\  right) computable ring $R$ and a finite free resolution of $M$ is \emph{constructible} \cite[Algorithm~1]{QR_S}. For a not necessarily commutative ring $R$ with finite global dimension there is yet another approach based on \textsc{Auslander}'s degree of torsion-freeness \cite{AB}, which involves the higher extension modules with values in $R$ of the so-called \textsc{Auslander} dual module of $M$. This approach is constructive if $R$ is (left \emph{and} right) computable and the finite global dimension of $R$ is explicitly known (cf.~\cite[Thm.~7]{CQR05}).

The \textsc{Quillen-Suslin} Theorem states that a polynomial ring $k[x_1,\ldots,x_n]$ over a principal ideal domain $k$ is \textsc{Hermite}, i.e., every stably free $k[x_1,\ldots,x_n]$-module is free. Algorithms were given in \cite{LS,LW,GV} and implementations in \cite{FQ} and \cite{CQR07}. A constructive version of \textsc{Stafford}'s Theorem \cite{stafford} offers a way to decide freeness and compute a free basis of finitely presented stably free modules over the \textsc{Weyl} algebra $A_n(k)$ and the rational \textsc{Weyl} algebra $B_n(k)$, where $k$ is a computable field of characteristic $0$ \cite{QR_S}. The \textsc{Jacobson} normal form \cite{cohn} offers an alternative algorithm for $B_1(k)$. An implementation can be found in the {\sf Maple} package \texttt{Janet} \cite{BlinkovCidetal2}.

\bigskip
Computability stands for deciding zero and computing syzygies. As mentioned in the Introduction, the axiomatic approach pursued so far underlines the conceptual importance of solving (in)homogeneous linear systems rather than computing a ``distinguished basis'', the latter being the traditional way to solve such systems.

\section{Computing over local commutative rings}\label{local rings}

This section provides a simple alternative approach to solve (in)homogeneous linear systems over localizations of commutative computable rings at maximal ideals. The simplicity lies in avoiding the computation of distinguished bases over these local rings. In particular, for localized polynomial rings this approach offers a way to circumvent \textsc{Mora}'s basis algorithm. More precisely: \\
For a commutative computable ring $R$ with a finitely generated maximal ideal $\m$ we show in Lemma \ref{LocalSyzygies} and Proposition \ref{LocalModuleMembership} how to reduce solving linear systems over the local ring $R_\m$ to solving linear systems over the ``global'' ring $R$. Thus, the computability of $R$ implies that of $R_\m$, and as a corollary, the computability of $R_\m-\mathbf{fpres}$ as an \textsc{Abel}ian category.

This reduction leads to a result in complexity analysis of local polynomial rings, formulated in §\ref{LocalComplexity}. Finally, in §\ref{comparison mora} we compare our approach to \textsc{Mora}'s algorithm in the case of localized polynomial rings.

Let $1\in S\subseteq R$ be a multiplicatively closed subset of the commutative ring $R$. Recall, the \textbf{localization} of $R$ at $S$ is defined by $S^{-1}R:=\lbrace\frac{r}{s}\mid r\in R, s\in S\rbrace{/\!\sim}$, where $\frac{r}{s}\sim \frac{r^\prime}{s^\prime}\Leftrightarrow 0\in (rs^\prime-sr^\prime)\cdot S$. The localization $S^{-1}M$ of an $R$-module $M$ is defined by $S^{-1}R\otimes_R M$ and the localization $S^{-1}\phi$ of a morphism $\phi$ maps $\frac{m}{s}\in S^{-1}M$ to $\frac{\phi(m)}{s}$. Note that localization is an \textbf{exact functor}. For a prime ideal $\mathfrak{p}\triangleleft R$ the set $S:=R\setminus \mathfrak{p}$ is multiplicatively closed and the localization $M_\mathfrak{p}$ \emph{at} $\mathfrak{p}$ is defined as $S^{-1}M$. In this section we will only treat the case when the prime ideal $\mathfrak{p}=\m$ is \textbf{maximal}.

\subsection{Computability in the category of finite presentations over localized rings}\label{local_basic_constructions}

\begin{theorem}\label{ThmLocal}
Let $R$ be a commutative computable ring and $\m=\langle m_1,\ldots, m_k\rangle$ a finitely generated maximal ideal in $R$. Then $R_\m$ is a computable ring.
\end{theorem}

The proof of this theorem will be given in §\ref{ProofLocal}. We first draw the following conclusion:

\begin{coro}\label{CoroLocal}
Let $R$ and $\m$ as in Theorem~\ref{ThmLocal}. Then $R_\m-\mathbf{fpres}$ is \textsc{Abel}ian and is computable as an \textsc{Abel}ian category with enough projectives.
\end{coro}

\begin{proof}
Objects and morphisms in $R_\m-\mathbf{fpres}$ are given by finite matrices with entries in the localized ring $R_\m$, being fractions of elements of $R$. Each such matrix can be rewritten as a fraction $\frac{\tA}{s}$ with a numerator matrix $\tA$ over $R$ and a single element $s\in S:=R\setminus\m$ as a common denominator\footnote{This leads to a simple data structure for matrices over $R_\m$, which will prove advantageous from the standpoint of a computer implementation (cf.~§\ref{LocalizeRingForHomalg}).}. Using this simplified data structure for matrices over $R_\m$ is mandatory for the proof of Theorem~\ref{ThmLocal}. The computability of $R_\m$, essential for the constructions (\ref{SyzygiesGenerators}), (\ref{DecideZeroEffectively}), and (\ref{FreeLift}), is the statement of Theorem~\ref{ThmLocal}. We still need to go through the remaining basic constructions listed in the proof of Theorem~\ref{R-fpres_basic_constructions} and adapt the required matrix operations to the new data structure:

Points (\ref{IdentityMatrix})-(\ref{ZeroMatrix}) are covered by taking the identity matrix (with $1$ as denominator), composing the matrices $\frac{\tA}{s}$ and $\frac{\tB}{t}$ to $\frac{\mathtt{AB}}{st}$, taking an empty matrix, and taking the zero matrix (with $1$ as denominator), respectively. Points (\ref{AddMat})-(\ref{UnionOfColumns}), stemming from the axioms of an additive category, are also easily seen to reduce to the corresponding constructions over the global ring $R$ after writing the involved pairs of matrices with common denominators. The presentation matrix and representation matrix of the corresponding natural epimorphism of the cokernel (\ref{UnionOfRows_IdentityMatrix}) are again given by stacking matrices after writing them with a common denominator and by taking the identity matrix, respectively. The colift (\ref{Colift}) as described above was trivial anyway. The presentation matrix of a free hull (\ref{FreeHull}) of a module $M$ generated by $r_0$ elements is again given by an empty matrix with $r_0$ columns, while the identity matrix still represents the natural epimorphism.
\end{proof}

\subsection{Proof of Theorem~\ref{ThmLocal}}\label{ProofLocal}

Let $R$ be a commutative computable ring. Further let $\m=\langle m_1,\ldots, m_k\rangle$ a finitely generated maximal ideal in $R$. For the computability of $R_\m$ we need to compute generating sets of syzygies and to (effectively) solve the submodule membership problem (cf.~Def.~\ref{computable}). This is the content of Lemma~\ref{LocalSyzygies} and Proposition~\ref{LocalModuleMembership}.

\subsubsection{\texorpdfstring{$\mathtt{(Relative)SyzygiesGenerators}$}{(Relative)SyzygiesGenerators}}\label{LocalSyzygiesGenerators}

A matrix of generating syzygies over $R$ is also a matrix of generating syzygies over $R_\m$:

\begin{lemm}[Syzygies]\label{LocalSyzygies}
  Let $\tA\in R_\m^{m\times n}$. Rewrite  $\tA=\frac{\tilde{\tA}}{\tilde a}$ with $\tilde{\tA}\in R^{m\times n}$ and $\tilde a\in R\setminus\m$. If $\tilde{\tX}\in R^{k\times m}$ is a matrix of generating syzygies for $\tilde{\tA}$, then the matrix $\tX:=\frac{\tilde{\tX}}{1}$ is a matrix of generating syzygies of $\tA$.
\begin{proof}
Starting from an exact sequence $R^{1\times k}\stackrel{\tilde{\tX}}{\to}R^{1\times m}\stackrel{\tilde{\tA}}{\to}R^{1\times n}$ \emph{exactness} of the localization functor yields an exact sequence $R_\m^{1\times k}\stackrel{\tilde{\tX}/1}{\to}R_\m^{1\times m}\stackrel{\tilde{\tA}/1}{\to}R_\m^{1\times n}$. Since multiplication with $\frac{1}{\tilde a}$ is an isomorphism, the sequence $R_\m^{1\times k}\stackrel{\tilde{\tX}/1}{\to}R_\m^{1\times m}\stackrel{\tA}{\to}R_\m^{1\times n}$ is also exact.
\end{proof}
\end{lemm}

This lemma is valid more generally for any multiplicatively closed set $S$. Of course, the result of the algorithm can have redundant generators and thus might be a non-minimal set of generating syzygies.

\begin{rmrk}\label{Nakayama}
\textsc{Nakayama}'s Lemma implies that a finitely presented $R_\m$-module $M=\coker\tM$ is given on a minimal set of generators, if and only if its presentation matrix $\tM$ is unit-free (cf.~\cite[Chap.~5,~Prop.~4.3]{CLO5}). An element $\frac{d}{n}\in R_\m$ is a unit iff the numerator $d$ is not contained in $\m$, which is an ideal membership problem in $R$. Some elementary matrix transformations can now be used to construct unit-free presentation matrices (see for example \cite[§2]{BR}). Another easy consequence of \textsc{Nakayama}'s Lemma is the characterization of minimal resolutions over $R_\m$ as the unit-free ones \cite[§19.1]{Eis}. Being able to detect units in the computed syzygies we can use a standard procedure\footnote{Cf.~\cite[§3.2.1]{BR}, for example.} to compute a unit-free and hence minimal resolutions over $R_\m$.
\end{rmrk}

\subsubsection{\texorpdfstring{$\mathtt{DecideZero(Effectively)}$}{DecideZero(Effectively)}}\label{LocalDecideZero}

The algorithm for effectively deciding zero is a bit more involved. As seen in the proof of Lemma \ref{LocalSyzygies}, denominators of matrices can be omitted since multiplication with them is an isomorphism. So without loss of generality we assume all denominators $1$. To ease the notation let $\mathtt{m}:=\begin{pmatrix}m_1\\ \vdots \\ m_k\end{pmatrix}\in R^{k\times 1}$ denote the column of generators of the maximal ideal $\m$.

\begin{prop}[Submodule membership]\label{LocalModuleMembership}
Let $\tA=\frac{\tilde{\tA}}{1}\in R_\m^{m\times n}$ and $\tb=\frac{\tilde{\tb}}{1}\in R_\m^{1\times n}$ with numerator matrices $\tilde{\tA}$ and $\tilde{\tb}$ \emph{over} $R$. There exists a row matrix $\ttt\in R_\m^{1\times m}$ with $\mathtt{tA}+\tb=0$ iff there exists a matrix $\tilde{\ts}\in R^{1\times (m+k)}$ satisfying
\begin{eqnarray}
\tilde{\ts}\begin{pmatrix}\tilde{\tA} \\ \mathtt{m}\tilde{\tb}\end{pmatrix}+\tilde{\tb}=0.
\label{LocalModuleMembershipEqn}
\end{eqnarray}
\begin{proof}
Write $\tilde{\ts}=\begin{pmatrix}\mathtt{y} & \mathtt{z}\end{pmatrix}$ with $\mathtt{y}\in R^{1\times m}$, $\mathtt{z}\in R^{1\times k}$.

If $\tilde{\ts}$ and thus $\mathtt{y}$ and $\mathtt{z}$ exist, we simply set $t:=\mathtt{zm}+1\not\in\m$ and $\ttt:=\frac{\mathtt{y}}{t}$. For the converse implication write $\ttt=\frac{\tilde{\ttt}}{t}$ with $\tilde{\ttt}\in R^{1\times m}$ and $t\in R\setminus\m$. Since $t\not\in\m$ it has an inverse $y$ modulo $\m$, i.e., there exist $z_1,\ldots,z_k\in R$ with $yt=z_1 m_1+\ldots+z_k m_k+1$. We conclude
\begin{eqnarray*}
  \mathtt{tA+B=0} & \Leftrightarrow & \tilde{\ttt}\tilde{\tA}+t\tilde{\tb}=\mathtt{0}\\
     & \Leftrightarrow & y \tilde{\ttt}\tilde{\tA}+\left(z_1 m_1+\ldots+z_k m_k\right)\tilde{\tb}+1\cdot\tilde{\tb}=\mathtt{0}\\
     & \Leftrightarrow & \mathtt{y}\tilde{\tA}+\mathtt{zm}\tilde{\tb}+\tilde{\tb}=\mathtt{0}
\end{eqnarray*}
for $\mathtt{y}:=y\tilde{\ttt}$ and $\mathtt{z}:=\begin{pmatrix}z_1&\dots&z_k\end{pmatrix}\in R^{1\times k}$.
\end{proof}
\end{prop}

This proof is \emph{constructive}. Another short but nonconstructive proof was suggested by our colleague \textsc{Florian Eisele} using \textsc{Nakayama}'s Lemma, which is also the most illuminating way to interpret Formula (\ref{LocalModuleMembershipEqn}). The proof given here is indeed a reduction to the effective submodule membership problem over $R$. We don't see a way to generalize this proof to non-maximal prime ideals.

Note that we formulate the proposition for a single row matrix $\tb$. For a multi-row matrix $\tB$ simply stack the results of the proposition applied to each row $\tb$ of $\tB$. Contrary to \textsc{Gr\"obner} basis methods, the proof of the proposition does \emph{not} provide a normal form\footnote{This is an instance where the name $\mathtt{DecideZero}$ makes more sense than $\mathtt{Reduce}$ or $\mathtt{NormalForm}$.}. Nevertheless the proof yields an \emph{effective} solution of the submodule membership problem (see §\ref{RightDivide}). Further note that at no step we need to compute any kind distinguished basis over the local ring. The advantages but also the drawbacks of avoiding such a basis will be discussed in §\ref{comparison mora}.

\begin{exmp}\label{localexample}
Let $R=k[x]$ for an arbitrary field $k$ and $\m=\langle x\rangle$. We want to compute $\ttt\in R_\m^{1\times1}$ with $\mathtt{tA+b=0}$ for $\tA=\tilde{\tA}:=x-x^2$ and $\tb=\tilde{\tb}:=x$ regarded as $1\times 1$-matrices. In this case the gcd of $\tilde{\tA}$ and $\tm\tilde{\tb}$ coincides with $\tilde{\tb}$ and the extended \textsc{Euclid}ian algorithm\footnote{This is a special case of the \textsc{Hermite} normal form algorithm applied to a one-column matrix.} directly yields the desired coefficients $\mathtt{y}$ and $\mathtt{z}$:
\begin{eqnarray*}
  (\underbrace{-1}_{\mathtt{y}})\cdot(\underbrace{x-x^2}_{\tilde{\tA}})+(\underbrace{-1}_{\mathtt{z}})\cdot(\underbrace{x}_{\tm}\cdot \underbrace{x}_{\tilde{\tb}}) = \underbrace{-x}_{-\tilde{\tb}} & \Leftrightarrow & (\underbrace{-1}_{\mathtt{y}})\cdot(\underbrace{x-x^2}_{\tilde{\tA}})+(\underbrace{-x+1}_{t=\mathtt{z}\tm+1})\underbrace{x}_{\tilde{\tb}} =0\\
  & \Leftrightarrow & \underbrace{\frac{1}{x-1}}_{\ttt=\frac{\mathtt{y}}{t}}(\underbrace{x-x^2}_{\tA})+\underbrace{x}_{\tb}=0
\end{eqnarray*}
\end{exmp}

\subsection{Complexity estimation for local polynomials rings}\label{LocalComplexity}

As \textsc{E.~Mayr} has shown in \cite{Mayr}, the ideal membership problem over the polynomial ring $\mathbb{Q}[x_1,\dots, x_n]$ is exponential space complete. Proposition  \ref{LocalModuleMembership} implies a result about the complexity of computations over the local polynomial ring $\mathbb{Q}[x_1,\dots, x_n]_{\langle x_1,\dots, x_n\rangle}$: It gives a polynomial time reduction from the ideal membership problem over the localized polynomial ring $\mathbb{Q}[x_1,\dots, x_n]_{\langle x_1,\dots, x_n\rangle}$ to the ideal membership problem over the polynomial ring $\mathbb{Q}[x_1,\dots, x_n]$.

\begin{coro}
The ideal membership problem over $\mathbb{Q}[x_1,\dots, x_n]_{\langle x_1,\dots, x_n\rangle}$ is solvable in exponential space.
\end{coro}

\subsection{A comparison with \textsc{Mora}'s algorithm}\label{comparison mora}

For a polynomial ring $R:=k[x_1,\ldots,x_n]$ over a computable field $k$ \textsc{Mora}'s algorithm \cite{Mora} makes the category $R_\m-\textbf{fpmod}$ for $\m=\langle x_1,\ldots,x_n\rangle$ computable. The algorithms suggested in Lemma~\ref{LocalSyzygies} and Proposition~\ref{LocalModuleMembership} describe our alternative approach to establish the computability of $R_\m$, whereas \textsc{Mora}'s algorithm can be seen as the classical way to this end.

Except for the use of a a different reduction method capable of dealing with a local term ordering\footnote{Local term ordering implies $x_i<1$ for all $i$.}, \textsc{Mora}'s algorithm proceeds exactly in the same way as \textsc{Buchberger}'s algorithm, where of course leading terms and $s$-polynomial depend on the chosen term ordering. And since \textsc{Mora}'s algorithm computes, as a by product, the leading ideal/module one can read off \textsc{Hilbert} series. \textsc{Mora}'s different reduction method minimizes the so-called ``ecart'' (which measures the distance of a polynomial from being homogeneous) by a kind of elimination procedure, which can be very expensive. This makes \textsc{Mora}'s reduction slower than \textsc{Buchberger}'s. For a modern treatment of the theory of local standard bases in polynomial rings we refer to \cite[§1.6, 1.7]{GP08}. A free implementation can be found in \textsc{Singular} \cite{singular310}.

Hence, we argue that our approach to localization which only requires an implementation of \textsc{Buchberger}'s algorithm is computationally superior to a comparable implementation of \textsc{Mora}'s algorithm.

\textsc{Mora}'s slower reduction is even dwarfed by yet another major issue: Unlike for small input, we experienced that \textsc{Mora}'s algorithm does not scale well enough when applied to large matrices. In comparison with the reduction given by Proposition \ref{LocalModuleMembership}, we observed that \textsc{Mora}'s algorithm creates larger units in $R_\m$, which in subsequent computations have to be interpreted as denominators. And when writing several matrices over $R_\m$ with a common denominator, these large units blow up the entries of the numerator matrices. Clearly, this makes succeeding computations much harder.

\begin{rmrk}\label{MoraHilbertSeries}
\textsc{Mora}'s algorithm is still indispensable when it comes to computing \textsc{Hil\-bert} series (and related invariants), which cannot be computed using our approach. Intermediate homological computations tend to become huge, even if the final result is typically \emph{much} smaller. Our approach is thus suited to get through the intermediate steps, while \textsc{Mora}'s algorithm can then be applied to the smaller result, e.g., to obtain invariants. This is demonstrated in Example \ref{IntersectionFormula}.
\end{rmrk}

\section{Implementation and Data Structures}\label{implementation}

As mentioned in the Introduction, this paper suggests a specification for implementing homological algebra of \textsc{Abel}ian categories (cf.~\cite{CoqSpi}). This specification is realized in the {\tt homalg} project \cite{homalg-project}.

We found the programming language of {\sf GAP4} ideally suited to realize this specification. {\sf GAP4} provides object oriented and (to some useful extent) functional programming paradigms, classical method selection, multi-dispatching, and last but not least so-called immediate and true-methods, which are extensively used to teach {\sf GAP4} how to avoid unnecessary computations by applying mathematical reasoning. All these capabilities build upon a type-system, which is as simple as possible and as sophisticated as needed for the purposes of high level computer algebra  \cite{BrLi}.

The abstract setting of \textsc{Abel}ian categories is implemented in the {\tt homalg} package \cite{homalg-package} according to §\ref{basic_constructions}. Only building upon the basic constructions as abstract operations, the implementation provides routines to compute (co)homology, derived functors, long exact sequences \cite{BR}, \textsc{Cartan-Eilenberg} resolutions, hyper-derived functors, spectral sequences (of bicomplexes) and the filtration they induce on (co)homology \cite{BaSF}, etc.

In order to use these routines for performing computations within a concrete \textsc{Abel}ian category, the latter only needs to provide its specific implementation of the basic constructions. The specifics of our implementation for the category $\Rfpmod$ of finitely presented modules over a computable ring $R$ were detailed in §\ref{computability}, utilizing the natural equivalence $\coker:\Rfpres \xrightarrow{\sim} \Rfpmod$.  More precisely, the proof of Theorem~\ref{R-fpres_basic_constructions} shows how this equivalence of categories is used to translate \emph{all} constructions in the \textsc{Abel}ian category $\Rfpmod$ to operations on matrices. Note that this translation is independent of the computable ring $R$, a point reflected in our implementation.

This allows the matrices over specific computable rings with all their operations to reside outside {\sf GAP4}, preferably in a system that has performant implementations of all the matrix operations mentioned in the proof of Theorem~\ref{R-fpres_basic_constructions}. In turned out that {\sf GAP4} does not need to know the content but only few characteristic information about the matrices created during the computations, minimizing the communication between {\sf GAP4} and the external system drastically. For further details the interested reader is referred to the documentation of the {\tt homalg} project \cite{homalg-project}.

\subsection{\texorpdfstring{$\mathtt{LocalizeRingForHomalg}$}{LocalizeRingForHomalg}}\label{LocalizeRingForHomalg}

The algorithms presented in §\ref{local rings} are implemented in a \textsf{GAP4}-package \cite{GAP4} \texttt{Localize\-Ring\-For\-Homalg} \cite{localizeringforhomalg}. The implementation is abstract in the sense that \emph{any} commutative computable ring $R$ supported by a computer algebra system to which the {\tt homalg} project \cite{homalg-project} offers an interface can be localized at any of its finitely generated maximal ideals $\m$, thus providing a new ring $R_\m$ for the {\tt homalg} project. The package \texttt{LocalizeRingForHomalg} additionally includes an interface to the implementation of \textsc{Mora}'s algorithm in \textsc{Singular} \cite{singular310}, which can alternatively be used to solve (in)homogeneous linear systems over $R_\m$, making (together with the matrix operations in the proof of Corollary~\ref{CoroLocal}) the \textsc{Abel}ian category $R_\m-\mathbf{fpmod}$ computable.

Our aim to reduce computations over $R_\m$ to ones over $R$ suggests the above used well-adapted data structure for matrices $\tA$ over $R_\m$: Write $\tA$ as a fraction $\frac{\tilde{\tA}}{s}$ with numerator matrix $\tilde{\tA}$ over $R$ and a single denominator $s\in R\setminus\m$. Lemma~\ref{LocalSyzygies} and Proposition~\ref{LocalModuleMembership}, which provide the key algorithms for this reduction, require at least common denominators for each row of the input matrices. But taking common denominators for each row (or column) of a matrix is not suited for matrix multiplication. We saw in the proof of Corollary~\ref{CoroLocal} that this data structure uses no more than computations of common denominators to realize the remaining basic operations for matrices over $R_\m$.

The computational aspects of fraction arithmetic are critical for performance issues since writing matrices with a common denominator blows up numerator matrices. It became efficient the moment we started using least common multiples (as far as they exist in $R$) for common denominators and for canceling fractions representing ring elements.

Solving the submodule membership problem for the same submodule and various different elements occurs very often in homological computations. For Proposition~\ref{LocalModuleMembership} this means that the matrix $\tilde{\tA}$ as part of $\begin{pmatrix}\tilde{\tA} \\ \mathtt{m}\tilde{\tb}\end{pmatrix}$ enters many computations with different rows $\tilde{\tb}$. Thus replacing $\tilde{\tA}$ by a distinguished basis (over $R$, if such a basis exists) is a minor optimization. But with $\tilde{\tA}$ being a distinguished basis one can first check if $\tilde{\tb}$ reduces to zero modulo $\tilde{\tA}$. If so, then the row vector $\mathtt{z}$, occurring in the proof of Proposition~\ref{LocalModuleMembership}, can be assumed zero. Hence, the row vector $\ttt$, occurring in the statement of the proposition, has $1$ as denominator. This heuristic succeeds remarkably often and prevents the creation of unnecessary denominators, which would propagate through the remaining computations.

\section{Examples}\label{examples}

The examples below are computed using several packages from the \texttt{homalg} project \cite{homalg-project}, all written in \textsf{GAP4} \cite{GAP4}. Here we use \textsc{Singular} \cite{singular310} as one of the most performant \textsc{Gr\"obner} basis engines with an existing interface in the {\tt homalg} project.

\begin{exmp}\label{Purity}
Let $R:=\Q[a,b,c,d,e]$ and $M$ the $R$-module given by 4 generators satisfying the 9 relations given below, i.e.,  presented by a $9\times4$-matrix $\tA$:
\begin{lstlisting}
gap> LoadPackage( "RingsForHomalg" );;
gap> R := HomalgFieldOfRationalsInSingular( ) * "a,b,c,d,e";
<An external ring residing in the CAS Singular>
gap> A := HomalgMatrix( "[\
> 2*a+c+d+e-2,2*a+c+d+e-2,2*a+c+d+e-2,0,\
> 2*c*d+d^2+2*c*e+2*d*e-3*e^2-2*c-d-6*e+1,\
>   2*c*d+d^2+2*c*e+2*d*e-3*e^2-2*c-d-6*e+1,\
>   2*c*d+d^2+2*c*e+2*d*e-3*e^2-2*c-d-6*e+1,0,\
> -4*a+2*b-c-d-e+2,-4*a+2*b-c-d-e+2,\
>   -c+d+e+2,4*a*d-2*b*d+2*d^2+2*d*e,\
> c^2-d-1,c^2-d-1,c^2-d-1,0,\
> 4*d*e^2-d^2+2*c*e+4*d*e-3*e^2-2*c+3*d-6*e+5,\
>   4*d*e^2-d^2+2*c*e+4*d*e-3*e^2-2*c+3*d-6*e+5,\
>   4*d*e^2-d^2+2*c*e+4*d*e-3*e^2-2*c+3*d-6*e+5,0,\
> 0,b^2+a+c+d+e,0,b^2*e+a*e+c*e+d*e+e^2,\
> 0,b^2*d+a*d+c*d+d^2+d*e,0,0,\
> 0,a*b^2+a^2+a*c+a*d+a*e,0,0,\
> 4*b^3*d-4*d^3-12*d^2*e-32*c*e^2+12*e^3+21*d^2\
>   -42*c*e+40*d*e+27*e^2+8*c-9*d+16*e-17,\
>   -4*a*b*d-4*b*c*d-4*b*d^2-4*d^3-4*b*d*e-12*d^2*e-32*c*e^2+\
>   12*e^3+21*d^2-42*c*e+40*d*e+27*e^2+8*c-9*d+16*e-17,\
>   -12*d^2*e-32*c*e^2+12*e^3+21*d^2-42*c*e+44*d*e+\
>   27*e^2+8*c-9*d+16*e-17,-4*b^3+4*d^2+4*e\
> ]", 9, 4, R );
<A 9 x 4 matrix over an external ring>
gap> LoadPackage( "Modules" );;
gap> M := LeftPresentation( A );
<A left module presented by 9 relations for 4 generators>
\end{lstlisting}

Let $R_0 := R_\m$ denote the localized ring at the maximal ideal $\m=\langle a,b,c,d,e \rangle$ in $R$ corresponding to the ``origin'' in $\mathbb{A}^5(\Q)$. We now want to compute (following \cite{BaSF}) the purity filtration (=equidimensional filtration) of the localized $R_\m$-module $$M_0:=M_\m=R_\m\otimes_R M.$$

One possible approach would be to compute the purity filtration for the global $R$-module $M$ and to localize the resulting filtration afterwards\footnote{Justifying this statement is left to the reader.}. But since the module $M$ is also supported\footnote{Recall, $\operatorname{supp} M := \{ \mathfrak{p} \in \operatorname{Spec}(R) \mid M_\mathfrak{p} \neq 0 \}$.} at components not including the origin it is clear that computing the global purity filtration will automatically accumulate any structural complexity of $M$ at these components as well. The algorithm suggested in \cite{BaSF} for computing the purity filtration starts by resolving $M$. And indeed, an early syzygy computation during the resolution of $M$ failed to terminate within a reasonable time.

The computation for the localized module $M_0$ over the local ring $R_0$ did not terminate using \textsc{Mora}'s algorithm either, where it gets stuck in a basis computation at an early stage.

However, the purity filtration can be computed for the localized module $M_0$ using the approach suggested in §\ref{local rings} within seconds:
\begin{lstlisting}
gap> LoadPackage( "LocalizeRingForHomalg" );;
gap> R0 := LocalizeAtZero( R );
<A local ring>
gap> M0 := R0 * M;
<A left module presented by 9 relations for 4 generators>
gap> ByASmallerPresentation( M0 );
<A left module presented by 10 relations for 3 generators>
gap> filt0 := PurityFiltration( M0 );
<The ascending purity filtration with degrees [ -3 .. 0 ] and graded parts:
   0:   <A zero left module>
  -1:   <A cyclic reflexively pure grade 1 left module presented by
1 relation for a cyclic generator>
  -2:   <A reflexively pure grade 2 left module presented by 7 relations for 
2 generators>
  -3:   <A cyclic reflexively pure grade 3 left module presented by
3 relations for a cyclic generator>
of
<A non-pure grade 1 left module presented by 10 relations for 3 generators>> 
\end{lstlisting}
The following command computes an explicit isomorphism between a new module (equip\-ped with a triangular presentation compatible with the purity filtration) and our original module $M_0$.
\begin{lstlisting}
gap> m := IsomorphismOfFiltration( filt0 );
<An isomorphism of left modules>
gap> FilteredModule := Source( m );
<A left module presented by 7 relations for 4 generators>
gap> Display( FilteredModule );
_[1,1],0,     0,     _[1,4],
0,     _[2,2],0,     0,
0,     0,     _[3,3],0,
0,     _[4,2],_[4,3],-b-1/2,
0,     0,     0,     _[5,4],
0,     0,     0,     _[6,4],
0,     0,     0,     _[7,4]
/(4*a*b*d^3-2*b^2*d^3+2*b*d^3*e+2*a*d^3-b*d^3+d^3*e-4*a*b*d+2*b^2*d+2*b*d^2-2*\
b*d*e-2*a*d+b*d+d^2-d*e-2*b-1)

Cokernel of the map

R^(1x7) --> R^(1x4), ( for R := Q[a,b,c,d,e]_< a, b, c, d, e > )

currently represented by the above matrix
\end{lstlisting}
\textsc{Singular} suppressed the relatively big entries of the triangular presentation matrix. We use the following command to make them visible. We see that in fact all fractions cancel except for one entry which retains a denominator:
\begin{lstlisting}
gap> EntriesOfHomalgMatrix( MatrixOfRelations( FilteredModule ) );;
gap> ListToListList( last, 7, 4 );
[ [ (b^2+a+c+d+e)/1, 0/1, 0/1, -1/1 ],
  [ 0/1, (2*a-b+d+e)/1, 0/1, 0/1 ],
  [ 0/1, 0/1, (b^3-d^2-e)/1, 0/1 ],
  [ 0/1, (2*b^2*d*e+b*d*e+2*c*d*e+d^2*e+d*e^2)/1,
      (2*a*b*e+b^2*e+2*b*c*e+2*b*d*e+2*d^2*e+2*b*e^2+a*e+c*e+d*e+3*e^2)/1,
      (-1/2)/(2*a*d^3-b*d^3+d^3*e-2*a*d+b*d+d^2-d*e-1) ],
  [ 0/1, 0/1, 0/1, d/1 ],
  [ 0/1, 0/1, 0/1, a/1 ],
  [ 0/1, 0/1, 0/1, (b^4-b^3*e-b*e+e^2)/1 ] ]
\end{lstlisting}

\begin{rmrk}
One would wonder why local computations (as in Example~\ref{Purity} above and Example~\ref{IntersectionFormula} below) using the approach suggested in §\ref{local rings} could be faster than the global ones, although the local syzygies computation in Lemma~\ref{LocalSyzygies} is nothing but a global syzygies computation and the effective solution of the local submodule membership problem in Proposition~\ref{LocalModuleMembership} is again nothing but the effective solution of an adapted global one. This has to do with the fact that the local ring $R_\m$ has many more units than $R$, leading to the structural simplifications mentioned in Remark \ref{Nakayama}. Furthermore, $\texttt{DecideZero}_{R_\m}(\tB,\tA)$ often equals zero even if $\texttt{DecideZero}_{R}(\tilde{\tB},\tilde{\tA})$ (for the corresponding numerator matrices $\tilde{\tB},\tilde{\tA}$) does not (cf.~§\ref{RightDivide}). For the geometric interpretation of the last statement let $\langle\tilde{\tC}\rangle$ denote the submodule of the free $R$-module generated by the rows of the matrix $\tilde{\tC}$. Then the maximal ideal $\m$ often enough does not lie in the support of the nontrivial $R$-subfactor module $\langle\tilde{\tB}\rangle/\langle\tilde{\tA}\rangle$, i.e., the $R_\m$-subfactor $\langle\tB\rangle/\langle\tA\rangle$ is trivial although its global counterpart $\langle\tilde{\tB}\rangle/\langle\tilde{\tA}\rangle$ is not. So one can roughly say that zero modules occur more frequently in local homological computations than in global ones, making the former faster, in general.
\end{rmrk}

\end{exmp}

\begin{exmp}\label{IntersectionFormula}
\textsc{Serre}'s intersection multiplicity formula of two ideals $I, J \triangleleft R$ at a prime ideal $\mathfrak{p}\triangleleft R$ (cf.~\cite[Thm.~A.1.1]{Har})
$$i(I, J; \mathfrak{p}) = \sum_i(-1)^i \mathrm{length}\left(\operatorname{Tor}^{R_\mathfrak{p}}_i(R_\mathfrak{p}/I_\mathfrak{p},R_\mathfrak{p}/J_\mathfrak{p})\right)$$
offers a nice demonstration of Remark~\ref{MoraHilbertSeries}.

Let $R:= \mathbb{F}_5[x,y,z,v,w]$ with maximal ideal $\mathfrak{p}=\m=\langle x,y,z,v,w\rangle$. We use the package {\tt LocalizeRingForHomalg} to define the two localized rings $R_0=S_0:=R_\m$. The ring $R_0$ utilizes Lemma~\ref{LocalSyzygies} and Proposition~\ref{LocalModuleMembership}, whereas $S_0$ uses \textsc{Mora}'s algorithm to solve (in)homogeneous linear systems.
\begin{lstlisting}
gap> LoadPackage( "RingsForHomalg" );;
gap> R := HomalgRingOfIntegersInSingular( 5 ) * "x,y,z,v,w";;
gap> LoadPackage( "LocalizeRingForHomalg" );;
gap> R0 := LocalizeAtZero( R );;
gap> S0 := LocalizePolynomialRingAtZeroWithMora( R );;
\end{lstlisting}

The ideals $I$ and $J$ are the intersection of ideals similar to those in \cite[Example.~A.1.1.1]{Har} with ideals not supported at zero:
\begin{lstlisting}
gap> i1 := HomalgMatrix( "[ \
> x-z, \
> y-w  \
> ]", 2, 1, R );;
gap> i2 := HomalgMatrix( "[ \
> y^6*v^2*w-y^3*v*w^20+1,         \
> x*y^4*z^4*w-z^5*w^5+x^3*y*z^2-1 \
> ]", 2, 1, R );;
gap> LoadPackage( "Modules" );;
gap> I := Intersect( LeftSubmodule( i1 ), LeftSubmodule( i2 ) );;
gap> I0 := R0 * I;
<A torsion-free left ideal given by 4 generators>
gap> OI0 := FactorObject( I0 );
<A cyclic left module presented by yet unknown relations for a cyclic generator>
gap> j1 := HomalgMatrix( "[ \
> x*z, \
> x*w, \
> y*z, \
> y*w, \
> v^2  \
> ]", 5, 1, R );;
gap> j2 := HomalgMatrix( "[ \
> y^6*v^2*w-y^3*v*w^2+1,           \
> x*y^4*z^4*w-z^5*w^5+x^3*y*z^2-1, \
> x^7                              \
> ]", 3, 1, R );;
gap> J := Intersect( LeftSubmodule( j1 ), LeftSubmodule( j2 ) );;
gap> J0 := R0 * J;;
gap> OJ0 := FactorObject( J0 );
<A cyclic left module presented by yet unknown relations for a cyclic generator>
\end{lstlisting}

Computing the $\mathrm{Tor}$-modules over $S_0$ (using \textsc{Mora}'s algorithm) or globally\footnote{$\operatorname{Tor}_i^{S^{-1} R}(S^{-1} M, S^{-1} N)=S^{-1}\operatorname{Tor}_i^R(M,N)$ for multiplicatively closed subsets $S\subset R$,  cf.~\cite[Prop.~7.17 or Cor.~10.72]{Rot09}.} over $R$ both did not terminate within a week. But over $R_0$ (using the approach suggested in §\ref{local rings}) the computation terminates in few seconds:
\begin{lstlisting}
gap> T0 := Tor( OI0, OJ0 );
<A graded homology object consisting of 3 left modules at degrees [ 0 .. 2 ]>
\end{lstlisting}

As mentioned in Remark §\ref{MoraHilbertSeries}, our approach cannot produce a ``distinguished basis'' for the presentation matrices of the resulting $\operatorname{Tor}$-modules, from which their \textsc{Hilbert} series can be read off. The sum of the coefficients of the \textsc{Hilbert} series of the module $\operatorname{Tor}^{R_\m}_i(R_\m/I_\m,R_\m/J_\m)$ is nothing but its dimension as a $R/\m\cong\mathbb{F}_5$ vector space, which in this case coincides with the $\operatorname{length}$. But now we can apply \textsc{Mora}'s algorithm to the already computed presentation matrices of the modules $\operatorname{Tor}^{R_\m}_i(R_\m/I_\m,R_\m/J_\m)$:
\begin{lstlisting}
gap> T0Mora := S0 * T0;
<A sequence containing 2 morphisms of left modules at degrees [ 0 .. 2 ]>
gap> List ( ObjectsOfComplex ( T0Mora ), AffineDegree );
[ 6, 2, 0 ]
\end{lstlisting}
Thus the intersection multiplicity at $\m$ is $6-2+0=4$.

Of course, getting rid of the irrelevant\footnote{I.e., those primary ideals $\mathfrak{q}$ with $\mathfrak{q}\not\subset \m$.} primary components of $I$ and/or $J$ would simplify the computations. But we were not able to compute a primary decomposition for $I$ with the computer algebra system \textsc{Singular}, and computing one for $J$ took more than seven minutes, which is still longer than the few seconds needed by our approach.
\end{exmp}

\section{Conclusion}

The above axiomatic setup for algorithmic homological algebra in the categories of finitely presented modules, merely viewed as \textsc{Abel}ian categories, only requires solving (in)homogeneous linear systems and does not enforce the introduction of any notion of distinguished basis. This abstraction motivated a constructive approach to the homological algebra of such module categories over commutative rings localized at maximal ideals, an approach in which global computations replace local ones.

For local polynomial rings this shows that \textsc{Buchberger}'s algorithm provides an alternative to \textsc{Mora}'s algorithm, an alternative which for some examples may even lead to a remarkable gain in computational efficiency. \textsc{Mora}'s algorithm remains indispensable when it comes to computing \textsc{Hilbert} series. It is hence a mixture of both algorithms that proves more useful in practice.

\section*{Acknowledgement}

We would like to thank Sebastian Posur for spotting a typo in the proof of Lemma~\ref{LocalSyzygies}.

\appendix

\section{Existential Quantifiers in \textsc{Abel}ian Categories}\label{abelian}

In this appendix we will briefly recall the \emph{existential quantifiers} for lift, colift, and projective lift appearing in §\ref{basic_constructions}.

Let $\A$ be a category. We use the following convention\footnote{This differs from the convention followed in \cite[§II.1, p.~41]{HS}.} for composing morphisms:
\[
\begin{array}{rclcc}
  \Hom_\A(M,L)&\times& \Hom_\A(L,N) &\to& \Hom_\A(M,N), \\
  (\phi&,&\psi) &\mapsto& \phi\psi
\end{array}
\]

\begin{defn}[{\cite[§II.6, p.~61]{HS}}]\label{lift_colift}
Let $\A$ be a category with $0$ and $\phi:M \to N$ a morphism.
\begin{enumerate}
  \item[(k)] A morphism $\kappa:K \to M$ is called ``the'' \textbf{kernel} of $\phi:M \to N$
    if
    \begin{enumerate}
      \item[(i)] $\kappa\phi=0$, and
      \item[(ii)] for all objects $L$ and all morphisms $\tau:L\to M$ with $\tau\phi=0$ there \fbox{exists} a \emph{unique} morphism $\tau_0:L\to K$,
        such that $\tau=\tau_0\kappa$. $\tau_0$ is called the \textbf{lift} of $\tau$ along $\kappa$.
    \end{enumerate}
    It follows from the uniqueness of the lift $\tau_0$ that $\kappa$ is a \emph{monomorphism}.
    \[
      \xymatrix{
        L \ar@{.>}[d]_{\tau_0} \ar[rd]_\tau \ar[rrd]^0 \\
        K \ar@{^{(}->}[r]_\kappa & M \ar[r]_\phi & N
      }
    \]
     $K$ is called ``the'' \textbf{kernel object} of $\phi$. Depending on the context $\ker\phi$ sometimes stands for the morphism $\kappa$ and sometimes for the object $K$.
  \item[(c)] A morphism $\epsilon:M \to C$ is called ``the'' \textbf{cokernel} of $\phi:M \to N$
  if
    \begin{enumerate}
      \item[(i)] $\phi\epsilon=0$, and
      \item[(ii)] for all objects $L$ and all morphisms $\eta:N\to L$ with $\phi\eta=0$ there \fbox{exists} a \emph{unique} morphism $\eta_0:C\to L$,
        such that $\eta=\epsilon\eta_0$. $\eta_0$ is called the \textbf{colift} of $\eta$ along $\epsilon$.
    \end{enumerate}
    It follows from the uniqueness of the colift $\eta_0$ that $\epsilon$ is an \emph{epimorphism}.
    \[
      \xymatrix{
        L \\
        C \ar@{.>}[u]^{\eta_0} & N\ar[ul]^\eta \ar@{>>}[l]^\epsilon & M  \ar[l]^\phi \ar[llu]_0
      }
    \]
   $C$ is called ``the'' \textbf{cokernel object} of $\phi$. Depending on the context $\coker\phi$ sometimes stands for the morphism $\epsilon$ and sometimes for the object $C$.
\end{enumerate}
\end{defn}

\begin{defn}\label{special_projective_lift}
  An object $P$ in a category $\A$ is called \textbf{projective}, if for each epimorphism
  $\epsilon:M \twoheadrightarrow N$ and each morphism $\phi:P\to N$ there \fbox{exists} a morphism $\phi_1:P\to M$ with
  $\phi_1\epsilon=\phi$.
  \[
    \xymatrix{
      P \ar[rd]^\phi \ar@{.>}[d]_{\phi_1}\\
      M \ar@{>>}[r]_\epsilon & N
    }
  \]
  We call $\phi_1$ a \textbf{projective lift} of $\phi$ along $\epsilon$.
\end{defn}

\begin{rmrk}\label{general_projective_lift}
A supposedly more general form of the projective lift is often used in \emph{\textsc{Abel}ian} categories. The assumption of $\epsilon$ being epic can be relaxed in the following way: Let $\A$ be an \textsc{Abel}ian category and $\epsilon:M\to N$ a morphism with $\epsilon:M \stackrel{\pi}{\twoheadrightarrow} I \stackrel{\iota}{\hookrightarrow} N$. According to the homomorphism theorem (cf.~{\cite[Prop.~II.9.6]{HS}},{\cite[Thm.~IX.2.1]{ML95}}) there exists an essentially unique decomposition of $\epsilon$ into an epic $\pi$ and a monic $\iota$. Further let $\beta:N\to L$ be a morphism with kernel $\iota$ (in other words, $M\xrightarrow{\epsilon}N\xrightarrow{\beta}L$ is an \textbf{exact sequence}), $P\in\A$ projective object, and $\phi:P\to N$ a morphism with $\phi\beta=0$. This last condition expresses that the \textbf{image subobject}\footnote{For details see \cite[Def.~of~suboject, p.~306]{Rot09}.} of $\phi$ is ``contained'' in the image subobject of $\epsilon$. It easily follows that there exists a \textbf{projective lift} $\phi_1$ along $\epsilon$ making the following diagram commutative:
\[
  \xymatrix{
    P \ar@{.>}[d]_{\phi_1} \ar[rd]_\phi \ar[rrd]^0 \\
    M \ar[r]_\epsilon & N \ar[r]_\beta & L
  }
\]
The projective lift $\phi_1$ is constructed in two steps:
\[
  \xymatrix{
    & P \ar@{.>}_{\phi_1}[dl] \ar@{.>}[d]_{\phi_0} \ar[rd]_\phi \ar[rrd]^0 \\
    M \ar@{>>}[r]_{\pi} & I \ar@{^{(}->}[r]_\iota & N \ar[r]_\beta & L
  }
\]
First construct $\phi_0$ as the lift of $\phi$ along the monomorphism $\iota$. $\phi_1$ is then the projective lift of $\phi_0$ along the epimorphism $\pi$.
\end{rmrk}

\input{Abelian.bbl}



\end{document}

%% file: Abelian.bbl
\newcommand{\etalchar}[1]{$^{#1}$}
\def\cprime{$'$} \def\cprime{$'$} \def\cprime{$'$} \def\cprime{$'$}
  \def\cprime{$'$}
\providecommand{\bysame}{\leavevmode\hbox to3em{\hrulefill}\thinspace}
\providecommand{\MR}{\relax\ifhmode\unskip\space\fi MR }
\providecommand{\MRhref}[2]{%
  \href{http://www.ams.org/mathscinet-getitem?mr=#1}{#2}
}
\providecommand{\href}[2]{#2}

%% file: Abelian.bbl
\begin{thebibliography}{BCG{\etalchar{+}}03}

\bibitem[AB69]{AB}
Maurice Auslander and Mark Bridger, \emph{Stable module theory}, Memoirs of the
  American Mathematical Society, No. 94, American Mathematical Society,
  Providence, R.I., 1969. \MR{MR0269685 (42 \#4580)}

\bibitem[Bar09]{BaSF}
Mohamed Barakat, \emph{Spectral filtrations via generalized morphisms},
  submitted (\href{http://arxiv.org/abs/0904.0240}{\texttt{arXiv:0904.0240}})
  (v2 in preparation), 2009.

\bibitem[BCG{\etalchar{+}}03]{BlinkovCidetal2}
Y.~A. Blinkov, C.~F. Cid, V.~P. Gerdt, W.~Plesken, and D.~Robertz, \emph{{The
  MAPLE Package ``Janet'': II. Linear Partial Differential Equations}},
  Proceedings of the 6th International Workshop on Computer Algebra in
  Scientific Computing, Sep. 20-26, 2003, Passau (Germany), 2003, pp.~41--54.

\bibitem[BL98]{BrLi}
Thomas Breuer and Steve Linton, \emph{The \textsf{GAP4} type system: organising
  algebraic algorithms}, ISSAC '98: Proceedings of the 1998 international
  symposium on Symbolic and algebraic computation (New York, NY, USA), ACM,
  1998, pp.~38--45.

\bibitem[BLH14a]{homalg-package}
Mohamed Barakat and Markus Lange-Hegermann, \emph{The $\mathtt{homalg}$ package
  -- {A} homological algebra $\mathsf{GAP4}$ meta-package for computable
  {A}belian categories}, 2007--2014,
  (\url{http://homalg.math.rwth-aachen.de/index.php/core-packages/homalg-package}).

\bibitem[BLH14b]{localizeringforhomalg}
\bysame, \emph{$\mathtt{LocalizeRingForHomalg}$ -- {L}ocalize {C}ommutative
  {C}omputable {R}ings at {M}aximal {I}deals}, 2009--2014,
  {\tiny\url{http://homalg.math.rwth-aachen.de/index.php/extensions/localizeringforhomalg}}.

\bibitem[BR08]{BR}
Mohamed Barakat and Daniel Robertz, \emph{$\mathtt{homalg}$ -- {A} meta-package
  for homological algebra}, J. Algebra Appl. \textbf{7} (2008), no.~3,
  299--317,
  (\href{http://arxiv.org/abs/math.AC/0701146}{\texttt{arXiv:math.AC/0701146}}).
  \MR{2431811 (2009f:16010)}

\bibitem[Buc06]{Buch}
Bruno Buchberger, \emph{An algorithm for finding the basis elements of the
  residue class ring of a zero dimensional polynomial ideal}, J. Symbolic
  Comput. \textbf{41} (2006), no.~3-4, 475--511, Translated from the 1965
  German original by Michael P. Abramson. \MR{MR2202562 (2006m:68184)}

\bibitem[CLO05]{CLO5}
David~A. Cox, John Little, and Donal O'Shea, \emph{Using algebraic geometry},
  second ed., Graduate Texts in Mathematics, vol. 185, Springer, New York,
  2005. \MR{MR2122859 (2005i:13037)}

\bibitem[Coh85]{cohn}
P.~M. Cohn, \emph{Free rings and their relations}, second ed., London
  Mathematical Society Monographs, vol.~19, Academic Press Inc. [Harcourt Brace
  Jovanovich Publishers], London, 1985. \MR{MR800091 (87e:16006)}

\bibitem[CQR05]{CQR05}
Fr{'e}d{'e}ric Chyzak, Alban Quadrat, and Daniel Robertz, \emph{Effective
  algorithms for parametrizing linear control systems over {O}re algebras},
  Appl. Algebra Engrg. Comm. Comput. \textbf{16} (2005), no.~5, 319--376,
  {(\url{http://www-sop.inria.fr/members/Alban.Quadrat/PubsTemporaire/AAECC.pdf})}.
  \MR{MR2233761 (2007c:93041)}

\bibitem[CQR07]{CQR07}
\bysame, \emph{Ore{M}odules: a symbolic package for the study of
  multidimensional linear systems}, Applications of time delay systems, Lecture
  Notes in Control and Inform. Sci., vol. 352, Springer, Berlin, 2007,
  {(\url{http://www.mathb.rwth-aachen.de/OreModules})}, pp.~233--264.
  \MR{MR2309473}

\bibitem[CS07]{CoqSpi}
Thierry Coquand and Arnaud Spiwack, \emph{Towards constructive homological
  algebra in type theory}, Calculemus '07 / MKM '07: Proceedings of the 14th
  symposium on Towards Mechanized Mathematical Assistants (Berlin, Heidelberg),
  Springer-Verlag, 2007, (\url{http://hal.inria.fr/inria-00432525/PDF/v2.pdf}),
  pp.~40--54.

\bibitem[DL06]{DL}
Wolfram Decker and Christoph Lossen, \emph{Computing in algebraic geometry},
  Algorithms and Computation in Mathematics, vol.~16, Springer-Verlag, Berlin,
  2006, A quick start using SINGULAR. \MR{MR2220403 (2007b:14129)}

\bibitem[Eis95]{Eis}
David Eisenbud, \emph{Commutative algebra with a view toward algebraic
  geometry}, Graduate Texts in Mathematics, vol. 150, Springer-Verlag, New
  York, 1995. \MR{MR1322960 (97a:13001)}

\bibitem[FQ07]{FQ}
Anna Fabianska and Alban Quadrat, \emph{Applications of the {Q}uillen-{S}uslin
  theorem in multidimensional systems theory}, H. Park et G. Regensburger
  (eds.), Gr{\"o}bner Bases in Control Theory and Signal Processing, Radon
  Series on Computational and Applied Mathematics 3, de Gruyter, 2007,
  pp.~23--106.

\bibitem[GAP17]{GAP4}
The GAP~Group, \emph{{GAP -- Groups, Algorithms, and Programming, Version
  4.8.7}}, 2017.

\bibitem[GP08]{GP08}
G.~Greuel and G.~Pfister, \emph{A {\bf {s}ingular} introduction to commutative
  algebra}, extended ed., Springer, Berlin, 2008, With contributions by Olaf
  Bachmann, Christoph Lossen and Hans Sch{\"o}nemann, With 1 CD-ROM (Windows,
  Macintosh and UNIX). \MR{MR2363237 (2008j:13001)}

\bibitem[GPS09]{singular310}
G.~Greuel, G.~Pfister, and H.~Sch\"onemann, \emph{{\sc Singular} {3-1-0}}, {A
  Computer Algebra System for Polynomial Computations}, Centre for Computer
  Algebra, University of Kaiserslautern, 2009,
  \href{http://www.singular.uni-kl.de}{\tt http://www.singular.uni-kl.de}.

\bibitem[GV02]{GV}
Jes{\'u}s Gago-Vargas, \emph{Constructions in {$R[x_1,\dots,x_n]$}:
  applications to {$K$}-theory}, J. Pure Appl. Algebra \textbf{171} (2002),
  no.~2-3, 185--196. \MR{MR1904477 (2003b:13014)}

\bibitem[Har77]{Har}
Robin Hartshorne, \emph{Algebraic geometry}, Springer-Verlag, New York, 1977,
  Graduate Texts in Mathematics, No. 52. \MR{MR0463157 (57 \#3116)}

\bibitem[{hom}17]{homalg-project}
{homalg~project~authors}, \emph{The $\mathtt{homalg}$ project -- {A}lgorithmic
  {H}omological {A}lgebra}, (\url{http://homalg-project.github.io}),
  2003--2017.

\bibitem[HS97]{HS}
P.~J. Hilton and U.~Stammbach, \emph{A course in homological algebra}, second
  ed., Graduate Texts in Mathematics, vol.~4, Springer-Verlag, New York, 1997.
  \MR{MR1438546 (97k:18001)}

\bibitem[KR00]{KR}
Martin Kreuzer and Lorenzo Robbiano, \emph{Computational commutative algebra.
  1}, Springer-Verlag, Berlin, 2000. \MR{MR1790326 (2001j:13027)}

\bibitem[Lam99]{lam99}
T.~Y. Lam, \emph{Lectures on modules and rings}, Graduate Texts in Mathematics,
  vol. 189, Springer-Verlag, New York, 1999. \MR{MR1653294 (99i:16001)}

\bibitem[Lev05]{LevThesis}
Viktor Levandovskyy, \emph{{Non-commutative Computer Algebra for polynomial
  algebras: Gr\"{o}bner bases, applications and implementation}}, Ph.D. thesis,
  University of Kaiserslautern, June 2005.

\bibitem[LS92]{LS}
Alessandro Logar and Bernd Sturmfels, \emph{Algorithms for the
  {Q}uillen-{S}uslin theorem}, J. Algebra \textbf{145} (1992), no.~1, 231--239.
  \MR{MR1144671 (92k:13006)}

\bibitem[LW00]{LW}
Reinhard~C. Laubenbacher and Cynthia~J. Woodburn, \emph{A new algorithm for the
  {Q}uillen-{S}uslin theorem}, Beitr\"age Algebra Geom. \textbf{41} (2000),
  no.~1, 23--31. \MR{MR1745576 (2001e:13013)}

\bibitem[May89]{Mayr}
Ernst Mayr, \emph{Membership in polynomial ideals over {$\mathbb{Q}$} is
  exponential space complete}, S{TACS} 89 ({P}aderborn, 1989), Lecture Notes in
  Comput. Sci., vol. 349, Springer, Berlin, 1989, pp.~400--406. \MR{MR1027419}

\bibitem[ML95]{ML95}
Saunders Mac~Lane, \emph{Homology}, Classics in Mathematics, Springer-Verlag,
  Berlin, 1995, Reprint of the 1975 edition. \MR{MR1344215 (96d:18001)}

\bibitem[Mor82]{Mora}
Ferdinando Mora, \emph{An algorithm to compute the equations of tangent cones},
  Computer algebra ({M}arseille, 1982), Lecture Notes in Comput. Sci., vol.
  144, Springer, Berlin, 1982, pp.~158--165. \MR{MR680065 (84c:13012)}

\bibitem[QR07]{QR_S}
Alban Quadrat and Daniel Robertz, \emph{Computation of bases of free modules
  over the {W}eyl algebras}, J. Symbolic Comput. \textbf{42} (2007), no.~11-12,
  1113--1141. \MR{MR2368075 (2009a:16041)}

\bibitem[Rob06]{RobPhd}
Daniel Robertz, \emph{{Formal Computational Methods for Control Theory}}, Ph.D.
  thesis, RWTH Aachen, Germany, June 2006, This thesis is available at
  \url{http://darwin.bth.rwth-aachen.de/opus/volltexte/2006/1586}.

\bibitem[Rot09]{Rot09}
Joseph~J. Rotman, \emph{An introduction to homological algebra}, second ed.,
  Universitext, Springer, New York, 2009. \MR{MR2455920 (2009i:18011)}

\bibitem[Ser55]{FAC}
Jean-Pierre Serre, \emph{Faisceaux alg\'ebriques coh\'erents}, Ann. of Math.
  (2) \textbf{61} (1955), 197--278. \MR{MR0068874 (16,953c)}

\bibitem[Sta78]{stafford}
J.~T. Stafford, \emph{Module structure of {W}eyl algebras}, J. London Math.
  Soc. (2) \textbf{18} (1978), no.~3, 429--442. \MR{MR518227 (80i:16040)}

\bibitem[Wei94]{weihom}
Charles~A. Weibel, \emph{An introduction to homological algebra}, Cambridge
  Studies in Advanced Mathematics, vol.~38, Cambridge University Press,
  Cambridge, 1994. \MR{MR1269324 (95f:18001)}

\bibitem[ZL02]{ZL}
E.~Zerz and V.~Lomadze, \emph{A constructive solution to interconnection and
  decomposition problems with multidimensional behaviors}, SIAM J. Control
  Optim. \textbf{40} (2001/02), no.~4, 1072--1086 (electronic). \MR{MR1882725
  (2002m:93020)}

\end{thebibliography}
